\RequirePackage{fix-cm}
\documentclass[smallcondensed]{svjour3}     
\smartqed  
\usepackage{graphicx}
\usepackage{amsmath,amssymb}
\renewcommand{\P}{\mathbb{P}}
\newcommand{\E}{\mathbb{E}}
\newcommand{\Rd}{\mathbb{R}^d}
\newcommand{\Rp}{\mathbb{R}^+}
\newcommand{\F}{\mathcal{F}}
\newcommand{\Lgen}{\mathcal{L}}
\newcommand{\barr}{\begin{array}{rcl}}
\newcommand{\earr}{\end{array}}

\newcommand{\disp}{\displaystyle}
\newcommand{\Ema}{\E^{m_a}}

\newcommand{\ind}{1{\hskip -2.5 pt}\hbox{I}}
\newcommand{\DomL}{\mathcal{D}(\Lgen)}

 \journalname{Potential Analysis}
\begin{document}

\title{Harnack Inequality and Regularity for a Product of Symmetric Stable Process and Brownian Motion}

\titlerunning{Harnack Inequalitiy}        

\author{Deniz Karl\i}


\institute{Deniz Karli \at
              Department of Mathematics,
	    The University of British Columbia, 1984 Mathematics Road
	    Vancouver, B.C.
	    Canada V6T 1Z2 \\
              \email{deniz.karli@math.ubc.ca}       }

\date{}

\maketitle

\begin{abstract}
In this paper, we consider a product of a symmetric stable process in $\Rd$ and a one-dimensional Brownian motion in $\Rp$. Then we define a class of harmonic functions with respect to this product process. We show that bounded non-negative harmonic functions in the upper-half space satisfy  Harnack inequality and prove that they are locally H\"older continuous. We also argue a result on Littlewood-Paley functions which are obtained by the $\alpha$-harmonic extension of an $L^p(\Rd)$ function.

 \subclass{60 \and 42}
\end{abstract}

\section{Introduction}
\label{intro}
In the last 20 years, there has been an increasing interest in non-continuous stochastic processes. In particular symmetric $\alpha$-stable processes play 
an important role in today's probability theory, and  there has
been  a remarkable increase in the number of applications of symmetric $\alpha$-stable processes. \\ 
In this paper, we focus on a product process $X_t$ which is the product of a d-dimensional symmetric $\alpha$-stable process and a one-dimensional Brownian motion. We define $\alpha$-harmonic functions with respect to this process in a probabilistic way and then study its applications.\\
The organization of this paper is as follows: \\
One of the most fundamental tools about harmonic functions is the Harnack inequality. In section \ref{HarnackInequality}, we prove a Harnack inequality for $\alpha$-harmonic functions in $\Rd\times\Rp$. First, we show that the hitting probability of a Borel set with positive measure is positive. Then we prove the Harnack inequality using Krylov-Safonov's approach \cite{krylov_safonov}. In section \ref{regularity}, we study regularity of these $\alpha$-harmonic functions and prove their H\"older continuity. \\
One of the ways to obtain an $\alpha$-harmonic function is to extend a function on $\Rd$ which can be taken as the boundary of $\Rd\times\Rp$. This is also known as the Dirichlet problem for the upper-half space. We define the solution of this problem by means of the semigroup corresponding to $X_t$. Then in the last section, we state some results on Littlewood-Paley functions studied by P.A. Meyer and N.Varopoulos and then prove a result on a partial Littlewood-Paley function by using Harnack inequality. \\ \\ \\


\section{Preliminaries}
In this section, we give the details of the setup and introduce the notation which will be used for the rest of the paper.
We consider the space $\Rd$ with d-dimensional Lebesgue measure $m$ and the upper-half space $\Rd \times \Rp $ as a subspace of  $ \mathbb{R}^{d+1} $. For simplicity, we identify the boundary of the upper-half plane, $\Rd \times \{0\}$, with $\Rd$. We will denote the product measure $m \otimes \epsilon_a $ by $m_a$ for $a \in \Rp$, where $\epsilon_a$ is point mass at $a$. Throughout the paper $c, c_1, c_2,...$ will denote constants. Their values may vary from line to line. 

Now we introduce our process in details. Let $Y_t$ denote a d-dimensional right-continuous, symmetric $\alpha$-stable process, that is, $Y_t$ is a right-continuous Markov process with independent and stationary increments and its characteristic function is $e^{-t|\xi |^\alpha}$. It is known that $Y_t$ satisfies the scaling property with the factor of $t^{1/\alpha}$, that is, the processes $(Y_{ct}-Y_0)$ and $c^{1/\alpha}(Y_{t}-Y_0)$ have the same distribution for any $c>0$. Let $P_t$ denote the semigroup corresponding to $Y_t$, that is, $P_t(f)(x)=\E^x(f(Y_t))$. Here $\P^x$ is a probability measure corresponding $Y_t$ started at the point $x$, and $\E^x$ is the expectation with respect to $\P^x$. It is known that the measure $ \P^x(Y_t \in dy)$ is absolutely continuous with respect to Lebesgue measure. We will denote transition densities of this symmetric stable process by $p(t,x,y)$. Unlike Brownian motion, there isn't any simple explicit formula for the transition densities. However, we will use the following estimate on $p(t,x,y)$ :
  \begin{align}\label{sas_estimate}
 		c_1\,(s^{-d/\alpha} \wedge \frac{s}{|x-y|^{d+\alpha}}) \leq p(s,x,y) \leq c_2\,(s^{-d/\alpha} \wedge \frac{s}{|x-y|^{d+\alpha}}) 
  \end{align}
for some positive constants $c_1$ and $c_2$. Moreover, by \cite[P.261]{sato} $p(s,0,x)$ can also be expressed as 
\begin{align}\label{trans_dens}
\int_0^\infty (4\pi u)^{-d/2} e^{-|x|^2/(4u)} g_{\alpha/2}(s,u)du
\end{align}
 where  $ g_{\alpha/2}(s,u)$ is the density of an $\alpha/2$ stable subordinator whose Laplace transform is given by $\int_0^\infty e^{-\lambda v}g_{\alpha/2}(s,v)dv=e^{-s\lambda^{\alpha/2}}$. Hence $p(s,0,x)$ is continuous and differentiable in the variable $x$. 

This stable process $Y_t$ forms the horizontal component of our product process if we think $\Rd$ as the horizontal and $\Rp$ as the vertical component of $\Rd\times\Rp$. On the vertical direction we let $Z_t$ be a one dimensional Brownian motion which is independent from $Y_t$, and so the product process is $X_t=(Y_t,Z_t)$.  We define the stopping time $T_0=\inf\{t \geq 0 : Z_t=0\}$, which is the first time $X_t$ hits the boundary which is $\Rd\times\{0\}$. We note that $T_0$ depends only on $Z_t$ by independence of $Y_t$ and $Z_t$. \\

In the classical context, the definition of a harmonic function can be given in both analytic and probabilistic way. We adopt the probabilistic definition and modify it to define $\alpha$-harmonic functions. In this paper, we call a continuous function $u(x,t)$ on $\Rd\times\Rp$ harmonic ($\alpha$-harmonic) if the process $u(X_{t\wedge T_0})$ is a martingale with respect to the filtration $\F_s=\sigma(X_{r \wedge T_0} ; r \leq s)$ and the probability measure $\P^{(x,t)}$ for any starting point $(x,t) \in \Rd\times\Rp$.  One natural way to obtain such a harmonic function is to start with a bounded Borel function $f$ on $\Rd$ and extend it to the upper-half space. One can write this extension by 
\begin{align}\label{extension}
\E^{(x,t)}f(Y_{T_0})=\int_0^\infty \E^x f(Y_s) \P^t(T_0 \in ds), \quad (x,t)\in \mathbb{R}^d\times\mathbb{R}^+.
\end{align}
Here, $\P^t(T_0 \in ds)$ is the exit distribution of Brownian motion and its explicit form is well known. It is the probability measure $\mu_t(ds)$ where 
\begin{align}\label{brownian_density}
	\mu_t(ds)=\frac{t}{2\sqrt{\pi}} e^{-t^2/4s} s^{-3/2} ds
\end{align}
(see \cite{meyer_long}). On the other hand, $\E^x f(Y_s)$ can be written as $P_s(f)(x)=\int f(y) p(s,x,y) \, dy.$ We note that this extension (\ref{extension}) is a continuous function of $(x,t)$ (by (\ref{brownian_density}) and (\ref{trans_dens})). For simplicity, let's denote this extension by $f$ as well and write $f_t(x)=f(x,t)=\E^{(x,t)}f(Y_{T_0})$  for $t>0$ and $f(x,0)=f(x)$. If we consider the process $M_t^f=f(X_{t\wedge T_0})$ and the filtration $\F_s=\sigma(X_{r \wedge T_0} ; r \leq s)$ then it is easy to see that $M_t^f$ is a martingale under the probability measure $\P^{(x,t)}$ for any $x$ and for any $t>0$ with respect to the given filtration. So the extension is a harmonic function.

This harmonic extension can also be expressed as a convolution $f(x,t)= f*q_t(x)$ where $q_t(x)=\int_0^\infty p(s,x,0) \mu_t(ds)$. If we define $Q_t$ by $Q_t=\int_0^\infty P_s \,\mu_t(ds)$ then we can see that the extension of $f$ is $f(x,t)=Q_t(f)(x)$.  We note that $Q_t$ satisfies semi-group properties. 

In section \ref{LpTheory}, the main idea is based on the relation between a deterministic integral and a probabilistic integral. For the conversion between these two integral, we will use two tools, namely the Green function of Brownian motion and L\'evy system formula for the jump terms. If $f$ is a positive Borel function, then  the Green function for Brownian motion is given by  
$$
\E^{a} \left[ \int_0^{T_0} f(Z_s) ds \right]=\int_0^\infty (z \wedge a) f(z) dz
$$
 for any a $\in \Rp$. 
The second tool, the L\'evy system formula, can be stated as follows. 
\begin{theorem}\label{Levy_system_corollary}
Suppose $f$ is a positive measurable function on $\Rd\times\Rd$. If $f(x,y)$ is zero on the diagonal then
$$ \E^x\left[ \sum_{s\leq t}f(Y_{s-},Y_s ) \right] = \E^x\left[\int_0^t  \, \int f(Y_s,Y_{s}+u) \,\frac{du}{|u|^{d+\alpha}} \, ds \right]$$
 for any $x\in \Rd$.  
\end{theorem}
For the proof, we refer to the paper \cite{Bass_Levin}. This property will be called as the L\'evy system formula throughout the rest of this paper.

 Now, let $\Lgen$ be the infinitesimal generator of the symmetric $\alpha$-stable process $Y_t$ on its domain $\DomL$, which is defined as  $\lim_{t\rightarrow 0^+} (P_t-I)/t$. It can be given explicitly by
$$
\Lgen f(x)=\int [f(x+h)-f(x)-\ind_{\{|h| \leq 1\}}\nabla f(x) \cdot h] \frac{c}{|h|^{d+\alpha}} dh \,
$$
\cite[section 3.3]{applebaum}. We know that the semi-group $P_t$ admits a square operator $\Gamma$ which is given by the relation
$$
   2\Gamma(f,g)=\Lgen(fg)-f\Lgen(g)-g\Lgen(f).
$$
If we denote the point evaluation of $\Gamma(f,g)$ at $x$  by $\Gamma_x(f,g)$, then we have the following proposition.

\begin{proposition} 
If $f$ is a bounded function in $\DomL$ and $f_t(x)=Q_tf(x)$ as defined above, then $$\Gamma_x(f_t,f_t)=\displaystyle c \int \frac{[f_t(x+h)-f_t(x)]^2}{|h|^{d+\alpha}} dh.$$
\end{proposition}

\begin{proof}
According to the paper \cite[Lemma 5 on p. 153]{meyer_long}, $f_t$ is in $\DomL$. Moreover, again by   \cite[Theorem 2 on p.147]{meyer_long}, $f_t^2$ is also in $\DomL$ since the map $x\rightarrow x^2$ is in $\mathcal{C}^2$. By definition , $2\Gamma(f_t,f_t)=\Lgen(f_t^2)-2f_t\Lgen(f_t)$. Now we need to consider two cases where $2>\alpha>1$ and $1\geq \alpha>0$ separately. In the first case, we have
$$
f_t(x)\Lgen(f_t(x))=\int [f_t(x)f_t(x+h)-f_t^2(x)-\ind_{\{|h| \leq 1\}}f_t(x)\nabla f_t(x) \cdot h] \frac{c}{|h|^{d+\alpha}} dh
$$
and
$$
\Lgen(f_t^2(x))=\int [f_t^2(x+h)-f_t^2(x)-\ind_{\{|h| \leq 1\}}\nabla (f_t^2)(x) \cdot h] \frac{c}{|h|^{d+\alpha}} dh \,.
$$
So,
\begin{align*}
2\Gamma_x(f_t,f_t)  & = \disp \int [f_t(x+h)-f_t(x)]^2 \frac{c}{|h|^{d+\alpha}} dh \\ 
                                 &    \qquad \disp +\int \ind_{\{|h| \leq 1\}}h \cdot [2f_t(x)\nabla f_t(x)-\nabla(f_t^2)(x)] \frac{c}{|h|^{d+\alpha}} dh \,. 
\end{align*}
Since $\nabla (f_t^2)=2f_t\nabla f_t$, the second term is zero. Hence the result follows. In the second case where $1\geq \alpha>0$ , we can drop the last term of the integrand of the generator, since the integral $\disp \int \ind_{\{|h| \leq 1\}} \frac{h}{|h|^{d+\alpha}} dh$ is zero. Then a similar calculation yields to the desired result.
\end{proof}
By Meyer \cite[P.158]{meyer_long}, the quadratic variation of the martingale $M_t^f$ is 
\begin{align}\label{quad_var}
\disp \langle M^f \rangle_t=2\int_0^{t\wedge T_0} g(Y_s,Z_s) ds
\end{align}
where $\disp g(x,t)=\Gamma_x(f_t,f_t)+\left[ \frac{\partial}{\partial t} f(x,t) \right]^2$ and $\Gamma_x(f_t,f_t)$ is as given above.  This function appear in section \ref{LpTheory} when we construct Littlewood-Paley functions.

Before we end this section, let's list some short remarks which will be useful later. 

\begin{remark}\label{remark:4}
 If $m$ is Lebesgue measure, let $\P^{m_a}$ be the measure defined by 
$$
  \P^{m_a}=\int \P^{(x,a)} \, m(dx)
$$
and $\disp \E^{m_a}$ the integral taken with respect to the measure $\P^{m_a}$.   \end{remark}
\begin{remark}\label{remark:5}
 The law of $X_{T_0}$ under the measure $\P^{m_a}$ is $m$.  \end{remark}
\begin{remark}\label{remark:6}
The semi-groups $P_t$ and $Q_t$ are invariant under integration with respect to Lebesgue measure. If $f$ is a bounded Borel function then 
$$
\int P_t(f)(x) \, m(dx) = \int f(x) \, m(dx)= \int Q_t(f)(x) \, m(dx).
$$
 \end{remark}




\section{Harnack Inequality}\label{HarnackInequality}

When one study harmonic functions on a domain, one of the most useful tools in harmonic analysis is the Harnack inequality. It allows us to compare values of a harmonic function inside a domain and it plays a crucial role for many applications. In this section, our main goal is to obtain a Harnack inequality (Theorem \ref{harnack_inequality}) in this setup which was described in the previous section. To prove the inequality, we follow Krylov-Safonov approach (see \cite{krylov_safonov}) and the method used by Bass-Levin in \cite{Bass_Levin}. The idea is based on the fact that the hitting probability of a Borel set with positive measure is non-zero. So we focus on hitting probability first.

Let's start by introducing the notation to be used in this section. Since we have different scaling factors in vertical and horizontal directions, we will consider the rectangular box $D_r(x,t)=\overrightarrow{D}_r(x) \times B_r (t)$, where
$$\overrightarrow{D}_r(x)=\{ y\in \Rd: |x_i-y_i|<\frac{r^{2/\alpha}}{2} ,\,i=1,...,d,\, x=(x_1,...,x_d)\}$$ 
and
$$B_r(t)=\{ s\in \mathbb{R}: |s-t|<\frac{r}{2}\}.$$ 
For $\epsilon\in (0,1)$, we define the box with $\epsilon$ margin by $D^\epsilon_r(x,t)=\overrightarrow{D}^\epsilon_r(x) \times B_{(1-\epsilon)r} (t)$ where
$$\overrightarrow{D}^\epsilon_r(x)=\{ y\in \Rd: |x_i-y_i|<\frac{(1-\epsilon^{2/\alpha})r^{2/\alpha}}{2} ,\,i=1,...,d,\, x=(x_1,...,x_d)\}.$$ 
We denote the first hitting time and first exit time of a Borel set A by $T_A$ and $\tau_A$, respectively. If A is a point set then we will use $T_x$ instead of $T_{\{x\}}$. 

First we observe that expected exit time from the box $D_r(x,t)$ is comparable to $r^2$.

\begin{lemma}
Let $\epsilon>0$. There exists a positive constant $c$ such that
$$
  \E^{(y,s)}(\tau_{D_r(x,t)})\geq\, c\, r^2
$$
for $(y,s)\in D_r^\epsilon (x,t)$ and $D_{2r}(x,t)\subset\Rd\times\Rp$.
\end{lemma}
\begin{proof}
Let $u>0$. Then using the independence of $Y_v$ and $Z_v$ first, and the scaling property second, we obtain

\begin{eqnarray*}
  \lefteqn{ \E^{(y,s)}(\tau_{D_r(x,t)}) } \\ 
		&\quad& \geq  \E^{(y,s)}(\tau_{D_r(x,t)};  \tau_{D_r(x,t)} \geq \epsilon^2 r^2 u)\\ 
		&\quad& \geq  \epsilon^2 r^2 u \, \P^{(y,s)}( \tau_{D_r(x,t)} \geq \epsilon^2 r^2 u)\\ 
  		&\quad& \geq \epsilon^2 r^2 u\, \P^{(y,s)}(\sup_{v\leq \epsilon^2r^2u} |Y_v-Y_0| <\frac{ \epsilon^{2/\alpha} r^{2/\alpha}}{2}, \, \sup_{v\leq \epsilon^2r^2u} |Z_v-Z_0| <\frac{ \epsilon r}{2}) \\ 
   		&& =  \epsilon^2 r^2 u\, \P^{y}(\sup_{v\leq \epsilon^2r^2u} |Y_v-Y_0| < \frac{\epsilon^{2/\alpha} r^{2/\alpha}}{2})\, \P^{s}( \sup_{v\leq \epsilon^2r^2u} |Z_v-Z_0| < \frac{\epsilon r}{2}) \\ 
		&& =   \epsilon^2 r^2 u\, \P^{0}(\sup_{v\leq u} |Y_{v}| < 1/2)\, \P^{0}( \sup_{v\leq u} |Z_{v}| < 1/2).
\end{eqnarray*}

If we choose $u$ small enough, then the last two probabilities above can be made bigger than $1/2$. Then the result follows.\end{proof}

\begin{lemma}\label{exp_leq_r2}
There exists a positive constant $c$ such that
$$
  \E^{(y,s)}(\tau_{D_r(x,t)})\leq\, c\, r^2
$$
for any $(y,s)\in D_r(x,t)$ and $D_{2r}(x,t)\subset\Rd\times\Rp$.
\end{lemma}
\begin{proof}
By scaling we may consider the case $r=1$, and so it is enough to show that $ \E^{(y,s)}(\tau_{D_1(x,t)})\leq\, c\,$ where $D_{2}(x,t)\subset\Rd\times\Rp$. Let $S$ be the first time when $Y_v$ jumps of size larger than two. If we use the L\'evy system formula, 
\begin{align*}
  \P^{(y,s)}(S\leq 1) & = \E^{(y,s)} \sum_{v\leq S\wedge 1}  \ind_{\{|Y_v-Y_{v-}|>2\}} = \E^{(y,s)} \int^{S\wedge 1}_0 \int_{|h|>2}  \frac{dh}{|h|^{d+\alpha}}\, dv\\ 
			     & =\left[ \int_{|h|>2}  \frac{dh}{|h|^{d+\alpha}} \right]\,\E^{(y,s)} ({S\wedge 1})\\ 
			     & \geq c\,\E^{(y,s)} ({S\wedge 1}; \, S>1) = c\,\P^{(y,s)} ( S>1)= c\,[1-\P^{(y,s)} ( S\leq 1)].		     
\end{align*}
Hence there is $c' \in (0,1)$ such that  $ \P^{(y,s)}(S\leq 1) \geq c'$, and thus $\P^{(y,s)}(S > 1) \leq (1-c')$. We note that $\tau_{D_1(x,t)}$ is smaller than $S$. Thus $\P^{(y,s)}(\tau_{D_1(x,t)} > 1) \leq (1-c')$. Let's denote the usual shift operator by $\theta_m$. By the Markov property,
\begin{align*}
	\P^{(y,s)}(\tau_{D_1(x,t)}>m+1) &= \P^{(y,s)}(\tau_{D_1(x,t)}>m, \tau_{D_1(x,t)}\circ \theta_m>1) \\ 
						  &= \E^{(y,s)}(  \P^{(y,s)}( \tau_{D_1(x,t)}\circ \theta_m>1 | \F_m) ; \tau_{D_1(x,t)}>m) \\ 
						   &= \E^{(y,s)}(  \P^{X_m}( \tau_{D_1(x,t)}>1) ; \tau_{D_1(x,t)}>m) \\ 
						   &\leq (1-c')\, \P^{(y,s)}( \tau_{D_1(x,t)}>m). 
\end{align*}
Then we obtain $\P^{(y,s)}(\tau_{D_1(x,t)}>m)\leq (1-c')^m$ by induction. This leads to 
$$
  \E^{(y,s)}(\tau_{D_1(x,t)}) \leq 1+\sum_{m=1}^\infty \P^{(y,s)}(\tau_{D_1(x,t)}>m) =c\,<\infty.
$$
\end{proof}

Before we prove our theorem on hitting probabilities for general Borel sets, we consider the simple sets, which are the boxes of the form $E\times [a,b]\subset \Rd\times\Rp$

\begin{theorem}\label{thm_box}
Suppose $(y,s)\in D_{2}(x,t)$ and $K=E\times [a,b]$ is a rectangular box in $D_1(x,t)$ such that $E\subset\mathbb{R}^d$. Then 
$$
   \P^{(y,s)}(T_K<\tau_{D_{3}(x,t)})\geq c \cdot m(E) \cdot (b-a)
$$
for some positive constant $c$ and $D_{6}(x,t)\subset\Rd\times\Rp$.
\end{theorem}

\begin{proof}
First, using the Levy system formula, we obtain
\begin{align*}
	 \P^{(y,s)}(T_K<\tau_{D_{3}(x,t)})& \disp \geq \E^{(y,s)}\left( \sum_{v\leq \tau_{D_3(x,t)}} \ind_{\{ Z_v\in [a,b] \}}  \ind_{\{ Y_v\neq Y_{v-} , Y_v \in E\}} \right)\\
			& = \E^{(y,s)}\left( \int_0^{\tau_{D_3(x,t)}} \ind_{\{ Z_v\in [a,b] \}}  \int_E \frac{dz}{|Y_v-z|^{d+\alpha}} dv\right)\\
			& \geq c\, m(E)\, \E^{(y,s)}\left( \int_0^{\tau_{D_3(x,t)}} \ind_{\{ Z_v\in [a,b] \}}dv\right).
\end{align*}
To verify the last inequality, notice that $|Y_v-z|$ is bounded above for any $z\in E$ under $\P^{(y,s)}$.\\
Next, we will show 
\begin{align*}
	\E^{(y,s)}\left( \int_0^{\tau_{D_3(x,t)}} \ind_{\{ Z_v\in [a,b] \}}dv\right) \geq c\, (b-a).
\end{align*}
In this proof, we will use $\tau^{(1)}_A$ and $T^{(1)}_A$ for the first exit time and the first hitting time of a Borel set $A\subset\mathbb{R}^d$ by symmetric stable process $Y_t$, respectively. Similarly, $\tau^{(2)}_A$ and $T^{(2)}_A$ denote the first exit time and the first hitting time of a Borel set $A\subset\mathbb{R}^+$ by Brownian motion $Z_t$, respectively. If $A\subset\Rp$ is a point set then we will use the notation $T^{(2)}_{\{u\}}=T^{(2)}_u$. \\
First, we note that $\tau_{D_3(x,t)}=  \tau^{(1)}_{\overrightarrow{D}_3(x)}\wedge  \tau^{(2)}_{B_3(t)}=\tau^{(1)}_{\overrightarrow{D}_3(x)}\wedge\tau^{(2)}_{(t-3/2,t+3/2)}$. Set $c_\alpha=(3^{2/\alpha}-2^{2/\alpha})/2$. Then 
$$  \tau_{c_\alpha}:= \inf\{v\geq 0 : |Y_v-Y_0|\geq c_\alpha\} \leq \tau^{(1)}_{\overrightarrow{D}_3(x)}$$
under $\P^{(y,s)}$, since we have $y \in  {\overrightarrow{D}_2(x)}$. Hence
\begin{align*}
	\E^{(y,s)}\left( \int_0^{ \tau_{{D}_3(x,t)}} \ind_{\{Z_v \in [a,b]\}}\, dv \right)\geq \E^{(y,s)}\left( \int_0^{\tau_{c_\alpha} \wedge \tilde \tau_3}\ind_{\{Z_v \in [a,b]\}}\, dv \right)
\end{align*}
where $\tilde \tau_3=T^{(2)}_{t+3/2}\wedge T^{(2)}_{t-3/2}$.
There are three possible cases: (i.) $s<a$, (ii.) $s>b$ or (iii.) $s\in [a,b]$. First, assume that $s<a$ and define a function $h:\mathbb{R} \rightarrow \mathbb{R}$ by

\begin{equation*}
h(x)= 
\begin{cases} 
	0 &\text{ if $x<a$,}\\
	(x-a)^2/2 &\text{if $x\in [a,b) $,}\\
	(b-a)(x-\frac{a+b}{2}) &\text{if $x\geq b $.}
\end{cases}
\end{equation*}
Note that $h\geq 0$ and $h''=\ind_{(a,b)}$, so that 
\begin{align*}
	\frac{1}{2}\E^{(0,s)}\left(\int_0^{\tau_{c_\alpha} \wedge \tilde \tau_3}     \ind_{\{Z_v \in [a,b]\}}\, dv  \right)&=\E^{(0,s)}(h(Z_{\tau_{c_\alpha} \wedge \tilde \tau_3}) -h(Z_0))\\
				&\geq \E^{(0,s)}(h(Z_{\tau_{c_\alpha} \wedge \tilde \tau_3}) -h(Z_0); T^{(2)}_{t+3/2}<T^{(2)}_{t-3/2}\wedge \tau_{c_\alpha})\\
				&= (h(t+3/2) -h(s)) \P^{(0,s)}(T^{(2)}_{t+3/2}<T^{(2)}_{t-3/2}\wedge \tau_{c_\alpha})\\
				&\geq (h(t+3/2) -h(s)) \P^{(0,t-1)}(T^{(2)}_{t+3/2}<T^{(2)}_{t-3/2}\wedge \tau_{c_\alpha})
\end{align*}
since $s \in (t-1,t+1)$.
Now we note that 
$$h(t+3/2)-h(s)\geq (b-a)(t+3/2-(b\vee s))\geq (b-a)$$
since  $[a,b]\subset (t-1/2,t+1/2)$.
Finally by translation invariance of Brownian motion 
$$\P^{(0,t-1)}(T^{(2)}_{t+3/2}<T^{(2)}_{t-3/2}\wedge \tau_{c_\alpha})=\P^{(0,1)}(T^{(2)}_{7/2}<T^{(2)}_{1/2}\wedge \tau_{c_\alpha})$$
and so 
$$ \E^{(0,s)}\left(\int_0^{\tau_{c_\alpha} \wedge \tilde \tau_3}     \ind_{\{Z_v \in [a,b]\}}\, dv  \right) \geq c(b-a). $$
This proves the the first part. \\

By symmetry, the case $s>b$ follows from the same argument. If we take $s'=2t-s$, $a'=2t-a$ and $b'=2t-b$ then 
\begin{align*}
	\E^{(0,s)}\left(\int_0^{\tau_{c_\alpha} \wedge \tilde \tau_3}     \ind_{\{Z_v \in [a,b]\}}\, dv  \right)=\E^{(0,s')}\left(\int_0^{\tau_{c_\alpha} \wedge \tilde \tau_3}     \ind_{\{Z_v \in [b',a']\}}\, dv  \right)
\end{align*}
by symmetry, and the last term is bounded below by $c(a'-b')=c(b-a)$.\\

For the last case, assume $s\in[a,b]$. We use the same function $h(x)$ as in the case $s<a$. Arguing as in the case $s<a$,
\begin{align*}
		\E^{(y,s)}\left(\int_0^{\tau_{D_3(x,t)}}     \ind_{\{Z_v \in [a,b]\}}\, dv  \right)
				&\geq \E^{(0,s)}\left(\int_0^{\tau_{c_\alpha} \wedge \tilde \tau_3}     \ind_{\{Z_v \in [a,b]\}}\, dv  \right)\\
				&=2\,\E^{(0,s)}(h(Z_{\tau_{c_\alpha} \wedge \tilde \tau_3}) -h(Z_0))\\
				&\geq 2\, (h(t+3/2) -h(s)) \P^{(0,s)}(T^{(2)}_{t+3/2}<T^{(2)}_{t-3/2}\wedge \tau_{c_\alpha}).\\
\end{align*}
Moreover, by translation invariance of Brownian motion
$$
	\P^{(0,s)}(T^{(2)}_{t+3/2}<T^{(2)}_{t-3/2}\wedge \tau_{c_\alpha})\geq \P^{(0,s)}(T^{(2)}_{s+3}<T^{(2)}_{s-1}\wedge \tau_{c_\alpha})=\P^{(0,2)}(T^{(2)}_{5}<T^{(2)}_{1}\wedge \tau_{c_\alpha}).
$$
Since $s\in [a,b]$, we have $h(s)\leq (b-a)^2/2$ and
\begin{align*}
	h(t+3/2)-h(s) &\geq (b-a) (t+3/2-\frac{a+b}{2})-\frac{(b-a)^2}{2}\\
				&\geq (b-a) (t+3/2-b)\\
				& \geq (b-a).
\end{align*}
Hence,  
\begin{align*}
		\E^{(y,s)}\left(\int_0^{\tau_{D_3(x,t)}}     \ind_{\{Z_v \in [a,b]\}}\, dv  \right) \geq \P^{(0,2)}(T^{(2)}_{5}<T^{(2)}_{1}\wedge \tau_{c_\alpha}) (b-a).
\end{align*}
So in any case
\begin{align*}
		\E^{(y,s)}\left(\int_0^{\tau_{D_3(x,t)}}     \ind_{\{Z_v \in [a,b]\}}\, dv  \right) \geq c (b-a).
\end{align*} 
\end{proof}

So we proved the hitting probability of a rectangular box is positive. We can extend this to any compact set in $D_1(x,t)$ using Krylov-Safonov's method which is based on covering the compact set with rectangular boxes. From here on, we denote the Lebesgue measure of a set $A$ in $\Rd\times\Rp$ by $|A|$.


\begin{corollary}\label{harnack_corollary}
There exists a non-decreasing function $\varphi:(0,1)\rightarrow (0,1)$ such that  if $A$ is a compact set inside $D_1(x,t)$ such that $|A|>0$ and $(y,s)\in D_{2}(x,t)$ then 
$$
   \P^{(y,s)}(T_A<\tau_{D_{3}(x,t)}) \geq \varphi( |A|)
$$
where $D_{6}(x,t)\subset\Rd\times\Rp$.
\end{corollary}
\begin{proof}
  First define 
  \begin{align}
  	\varphi(\epsilon) & =\inf\{\P^{(z,u)}(T_B<\tau_{D_3{(z_0,u_0)}}): (z_0,u_0)\in \Rd \times \Rp,  (z,u) \in D_2{(z_0,u_0)},\\ 
		&  |B|\geq \epsilon \, |D_1{(z_0,u_0)}|,\, B \subset D_1{(z_0,u_0)}, B \mbox{ is compact}, D_6{(z_0,u_0)}\subset \Rd\times\Rp \}. \nonumber
  \end{align}
  Set $\disp q_0=\inf_{\varphi(\epsilon)>0} \epsilon$. We claim that $q_0=0$. Suppose not, that is, let's assume that $q_0$ is strictly positive. This will give us a contradiction. First we note that $q_0<1$. So we can find $q$ such that $(q+q^2)/2<q_0<q$. Let $\eta=(q-q^2)/2$ so that $q-\eta=(q+q^2)/2$. Let $\rho>0$ which will be chosen later. There is $D_6(z_0,u_0) \subset \Rd \times \Rp$, $(z,u)\in D_2(z_0,u_0)$, and $B\in D_1(z_0,u_0)$ such that 
  \begin{align}
  	q>\frac{|B|}{|D_1(z_0,u_0)|} > q-\eta
  \end{align}
  and
  \begin{align}\label{cont1}
   	\P^{(z,u)}(T_B<\tau_{D_3(z_0,u_0)})<\rho \cdot \varphi(q)^2.
  \end{align}
  Without loss of generality we may drop $(z_0,u_0)$ and  denote $D_i(z_0,u_0)$ by $D_i$ for $i=1,2,3$. 
  Now we construct the rectangular region $D$ which is described in the Proposition \cite[7.2]{bass_diff_and_elliptic_op}. Since $B\subset D_1$ and $\disp q> \frac{|B|}{|D_1|}=|B|$, there exists a rectangular region $D$ which satisfies the following conditions:
  \begin{enumerate}
   \item $\disp D=\cup_i \hat{R}_i$ where $R_i$ is a cube and $\hat{R}_i$ denotes the cube with the same center as $R_i$ but side lengths three times as long, 
   \item the interiors of $R_i$'s are pairwise disjoint,
   \item $|B|\leq q \cdot |D \cap D_1|$, and
   \item $|B\cap R_i|>q\cdot |R_i|$ for all $i$.
   \end{enumerate}
 Since $|B|>q-\eta$ and $|B|\leq q\cdot |D\cap D_1|$, we have 
 \begin{align*}
 	|D\cap D_1| \geq \frac{|B|}{q} > \frac{q-\eta}{q} = \frac{q+1}{2} >q.
 \end{align*}
 Set $\tilde{D}=D\cap D_1$. This results in $|\tilde{D}|>q$. By the definition of $\varphi$ , we obtain
 \begin{align}
 	\P^{(z,u)}(T_{\tilde{D}}<\tau_{D_3}) > \varphi(q).
 \end{align}
 To obtain a contradiction we want to show that if $(z',u')\in \tilde{D}$ then 
 \begin{align}\label{cont2}
 	\P^{(z',u')}(T_B<\tau_{D_3})\geq \rho \cdot \varphi(q).
 \end{align}
 If this is true then by using the strong Markov property
 \begin{align*}
 	\P^{(z,u)}(T_B<\tau_{D_3}) & \geq \P^{(z,u)}(T_{\tilde{D}}<T_B<\tau_{D_3}) \\ 
						& \geq \E^{(z,u)}  \left(  \P^{X_{T_{\tilde{D}}}}   (T_B<\tau_{D_3}) ; \, T_{\tilde{D}} < \tau_{D_3} \right) \\ 
						& \geq \rho\cdot \varphi(q)\cdot \P^{(z,u)}  (  T_{\tilde{D}} < \tau_{D_3} ) \\ 
						&\geq \rho\cdot \varphi(q)^2.
 \end{align*}
 This contradicts (\ref{cont1}).
 To obtain (\ref{cont2}), first set $F_i=R_i\cap D_1$. Note that $F_i$ is a rectangular box in $D_1$. So $\hat{F}_i\subset D_3$. Moreover $F_i \cap B= R_i \cap B$, since $B\in D_1$. So we have $|F_i \cap B| > q \cdot |F_i| $, and if $(z',u')\in F_i$ then 
 \begin{align*}
 	\P^{(z',u')}(T_{F_i \cap B}<\tau_{\tilde{F}_i})>\varphi(q)
 \end{align*}
 by the definition of $\varphi$.
 
 If $(z',u')\in\tilde{D}$ then $(z',u')\in \hat{R}_i \cap D_1$ for some $R_i$. We can find a cube $K\subset \Rd$ and $K\times [c,d]\subset F_i$. By the Theorem \ref{thm_box}, there is $\rho$ such that 
 
 \begin{align*}
     \P^{(z',u')}(T_{F_i}<\tau_{D_3})\geq\P^{(z',u')}(T_{K\times [c,d]}<\tau_{D_3})>\rho.
 \end{align*}
 
 Finally, the strong Markov property implies that if $(z',u')\in \tilde{D}$, then 
 
 \begin{align*}
 	\P^{(z',u')}(T_B<\tau_{D_3}) & \geq \P^{(z',u')}(T_{F_i}<\tau_{D_3},\, T_{F_i \cap B} < \tau_{\tilde{F}_i}) \\ 
					&= \E^{(z',u')} \left(\P^{(z',u')}( \left. T_{F_i \cap B} < \tau_{\tilde{F}_i} \right| T_{F_i})  ;  T_{F_i}<\tau{D_3}\right)\\ 
					&=\E^{(z',u')} \left(\P^{X_{T_{F_i}}}(  T_{F_i \cap B} < \tau_{\tilde{F}_i} )  ;  T_{F_i}<\tau{D_3}\right)\\ 
					&=\varphi(q)\cdot \P^{(z',u')} \left( T_{F_i}<\tau{D_3}\right)\\ 
					&=\rho \cdot \varphi(q).
 \end{align*}
 This completes the proof.
\end{proof}


\begin{lemma}\label{harnack_lemma_1}
  Suppose $H$ is a function which is bounded, non-negative on $\mathbb{R}^d\times\mathbb{R}^+$ and supported in $(D_{2r}(x,t))^c$ where $D_{4r}(x,t)\subset\Rd\times\Rp$. If $(y,s),\,(y',s')\in D_r^\epsilon(x,t)$ then 
  $$
    \E^{(y,s)}H(X_{\tau_{D_r}(x,t)}) \leq c\, \E^{(y',s')}H(X_{\tau_{D_r}(x,t)}).
  $$ 
\end{lemma}
\begin{proof}
  It suffices to prove the statement for $H=\ind_F$ with $F\subset (D_{2r}(x,t))^c$. Since $Z_v$ is continuous , $H(X_{\tau_{D_r}})$ is non-zero if and only if $Y$ jumps from $D_r(x,t)$ into $(D_{2r}(x,t))^c$. Since $D_r(x,t)= \overrightarrow{D}_r(x) \times B_r(t)$, it is enough to consider the sets $F=E\times [c,d]$ with $[c,d]\subset B_r(t)$ and the function $H$ of the form $H=\ind_F$. Then
   \begin{align*}
           \E^{(y,s)}\ind_{E\times [c,d]}(X_{u \wedge \tau_{D_r(x,t)}}) & = \E^{(y,s)} \sum_{v\leq u\wedge \tau_{D_r(x,t)}}  \ind_{\{Y_v\not= Y_{v-}, Y_v \in E\}} \\ 
           										    & = \E^{(y,s)} \int_0^{ u\wedge \tau_{D_r(x,t)}} \int_E \frac{dh}{|h-Y_v|^{d+\alpha}} \,dv 
    \end{align*}
Since $Y_v \in \overrightarrow{D}_r(x)$ and $h\in E$, $|h-Y_v|$ is comparable to $|h-x|$, that is, there are constants $c_1$ and $c_2$ so that
$$
 c_1 \int_E \frac{dh}{|h-x|^{d+\alpha}} \leq \int_E \frac{dh}{|h-Y_v|^{d+\alpha}} \leq c_2 \int_E \frac{dh}{|h-x|^{d+\alpha}}.
$$
So 
   \begin{align*}
          \E^{(y,s)} \int_0^{ u\wedge \tau_{D_r(x,t)}} \int_E \frac{dh}{|h-Y_v|^{d+\alpha}} \,dv & \leq c_2 \int_E \frac{dh}{|h-x|^{d+\alpha}} \E^{(y,s)}(u\wedge \tau_{D_r(x,t)}),\\
    \end{align*}
and similarly
   \begin{align*}
          \E^{(y,s)} \int_0^{ u\wedge \tau_{D_r(x,t)}} \int_E \frac{dh}{|h-Y_v|^{d+\alpha}} \,dv & \geq c_1 \int_E \frac{dh}{|h-x|^{d+\alpha}} \E^{(y,s)}(u\wedge \tau_{D_r(x,t)}),\\
    \end{align*}
    Now if we let $u\rightarrow \infty$, and use the fact that $c r^2\leq \E^{(y,s)}( \tau_{D_r(x,t)}) \leq c' r^2 $, then 
   \begin{align*}
           c_3 \, r^2\, \int_E \frac{dh}{|h-x|^{d+\alpha}} \leq \E^{(y,s)}\ind_{E\times [c,d]}(X_{\tau_{D_r(x,t)}}) \leq c_4 \, r^2\, \int_E \frac{dh}{|h-x|^{d+\alpha}}.  
    \end{align*}    
Observe that the above inequality also holds under $\P^{(y',s')}$. Hence the result follows.
 \end{proof}


So far we have showed that any compact set with a positive measure is visited by $X_t$ with positive probability given that the starting point is close. To prove the Harnack inequality, we will use this fact by defining a compact set on which the harmonic function takes large values. Since this set is visited with positive probability, one can define a sequence on which the function is unbounded. 

To simplify our notation, we fix a point temporarily and we will denote the usual rectangular box around this center by $\widetilde{D}_r$.

\begin{theorem}\label{harnack_inequality}
	There exists $c>0$ such that if $h$ is non-negative and bounded on $\mathbb{R}^d\times \mathbb{R}^+$, harmonic in $\widetilde{D}_{16}$ and $\widetilde{D}_{32}\subset\Rd\times\Rp$ then 
	\begin{align*}
			h(y,t)\leq c\, h(y',t')\,, & \quad (y,t),\,(y',t')\in \widetilde{D}_1.
	\end{align*}
\end{theorem}
\begin{proof}
By taking a constant multiple of $h$, we may assume that $\inf_{\widetilde{D}_1}h=1/2$. Then there is $(y_0,t_0)\in \widetilde{D}_1$ such that $h(y_0,t_0)<1$. We will show that  $h$ is bounded above in $\widetilde{D}_1$ by a constant not depending on $h$. We show that if $h$ takes  a very large value in $\widetilde{D}_1$, then we can find a sequence in $\widetilde{D}_2$ for which $h$ is unbounded.  \\
 Let $(x,s)\in \widetilde{D}_1$ such that $h(x,s)=K$. $K$ will be chosen later and we can take $K$ large. Let $\epsilon>0$ be small. By the Lemma (\ref{harnack_lemma_1}), there is $c_2$ such that if $(x,s)\in \mathbb{R}^d\times \mathbb{R}^+$, $r>0$, $D_{r}$ denotes the rectangular box $D_r(x,s)$ about the center $(x,s)$, $D^\epsilon_{r}$ denotes the rectangular box with $\epsilon$ margin $D^\epsilon_r(x,s)$, $H$ is bounded, non-negative and supported in $D_{2r}^c$ then for any $(y,t),\,(y',t')\in D^\epsilon_r$
\begin{align}
	 \E^{(y',t')}H(X_{\tau_{D_r}}) & \leq c_2 \, \E^{(y,t)}H(X_{\tau_{D_r}}).
\end{align}
By Theorem (\ref{thm_box}),  
\begin{align}\label{harnack_eq_1}
	\P^{(y_0,t_0)}(T_{{D}_{2r/3}}<\tau_{\widetilde{D}_{16}})&\geq c_3 \, |{D}_{2r/3}|.
\end{align}

By Corollary (\ref{harnack_corollary}), there is $c_4$ such that if $C\subset D_{r/3}$ is a compact set with $\displaystyle \frac{|C|}{|D_{r/3}|}\geq \frac{1}{3}$, then 
\begin{align}\label{harnack_eq_11}
	\P^{(y,t)}(T_C<\tau_{D_r})&\geq c_4 \quad (y,t) \in D_{2r/3}.
\end{align}

Now let $\disp \eta=\frac{c_4}{3}$, and $\disp \zeta=\frac{1}{3}\wedge (c_2^{-1}\eta)$. Let $r$ be so that $\disp |D_{2r/3}|=\frac{2}{c_3\, c_4 \,\zeta K}$. Note that in this case $\disp r^{\frac{2d}{\alpha}+1}=c_5\,K^{-1}$, and so $\disp r=c_5'\,K^{-\alpha/{(2d+\alpha)}}$. Here we note that by taking a large enough value for $K$, we keep $r$ small (that is $r<\frac{1}{32}$) and so the rectangular boxes stay in the upper half space.\\ \\
Let  $A'=\{ (u,v)\in D_{r/3}: h(u,v)\geq \zeta K \}$. We claim that 
\begin{align*}
	\frac{|A'|}{|D_{r/3}|} &\leq\frac{1}{2}. 
\end{align*}
If not, then there is a compact subset $A$ such that 
\begin{align*}
	\frac{|A|}{|D_{r/3}|} &\geq\frac{1}{3}. 
\end{align*}
 Then (\ref{harnack_eq_1}) and (\ref{harnack_eq_11}) imply that 
\begin{align}\label{measure_dominance}
	1\geq\,h(y_0,t_0) & \geq \E^{(y_0,t_0)}[h(X_{T_A\wedge \tau_{\widetilde{D}_{16}}});\, T_A < \tau_{\widetilde{D}_{16}}] \notag  \\ 
				   &  \geq \zeta \, K \,\E^{(y_0,t_0)}\left[ \P^{X_{T_{D_{2r/3}}}}[ T_A < \tau_{\widetilde{D}_{16}}] \,;T_{D_{2r/3}} < \tau_{\widetilde{D}_{16}}\right]\notag\\
				&  \geq \zeta \, K \,\E^{(y_0,t_0)}\left[ \P^{X_{T_{D_{2r/3}}}}[ T_A < \tau_{{D}_{r}}] \,;T_{D_{2r/3}} < \tau_{\widetilde{D}_{16}}\right]\\ 
				&  \geq \zeta \, K \,c_4\,\P^{(y_0,t_0)}\left[ T_{D_{2r/3}} < \tau_{\widetilde{D}_{16}}\right]\notag\\  
				   &  \geq \zeta \, K \,c_3\,c_4\, |{D}_{2r/3}|=2.\notag 
\end{align}
This contradiction shows that 
\begin{align*}
	\frac{|A'|}{|D_{r/3}|} &\leq\frac{1}{2}. \\
\end{align*}
Let $C$ be a compact set in $D_{r/3} - A'$ such that 
$$
	\frac{|C|}{|D_{r/3}|} \geq \frac{1}{3} .
$$
Let $H=\ind_{D^c_{2r}}h$. If $\E^{(x,s)}[h(X_{\tau_{D_r}}); X_{\tau_{D_r}} \not\in D_{2r}]>\eta K$ then for any $(y,t) \in D_{r/3}$
\begin{align*}
	h(y,t) & =\E^{(y,t)}[h(X_{\tau_{D_r}})]\geq  \E^{(y,t)}[h(X_{\tau_{D_r}}); X_{\tau_{D_r}} \not\in D_{2r}] =\E^{(y,t)}[H(X_{\tau_{D_r}})]\\ 
		 &\geq c_2^{-1} \, \E^{(x,s)}[H(X_{\tau_{D_r}})]\geq c_2^{-1} \eta \, K \geq \zeta \,K.
\end{align*}
But this is a contradiction, since $A'$ is not all of $D_{r/3}$. So we must have $$\E^{(x,s)}[h(X_{\tau_{D_r}}); X_{\tau_{D_r}} \not\in D_{2r}]\leq\eta K.$$
Let's denote the supremum of $h$ on $D_{2r}$ by $M$. Then
\begin{align*}
	K&=h(x,s) =\E^{(x,s)}[h(X_{T_C\wedge \tau_{D_r}})] \\ 
	               &=\E^{(x,s)}[h(X_{T_C}); T_C< \tau_{D_r}] + \E^{(x,s)}[h(X_{ \tau_{D_r}});\, T_C\geq \tau_{D_r}, X_{\tau_{D_r}} \in D_{2r}] \\ 			&\quad + \E^{(x,s)}[h(X_{ \tau_{D_r}});\, T_C\geq \tau_{D_r}, X_{\tau_{D_r}} \not\in D_{2r}]  \\ 
	               &\leq \zeta\,K\, \P^{(x,s)}[ T_C< \tau_{D_r}]  + M \, \P^{(x,s)}[ \tau_{D_r}\leq T_C]+\eta\,K \\ 
	               &\leq \zeta\,K\, \P^{(x,s)}[ T_C< \tau_{D_r}]  + M \, (1-\P^{(x,s)}[ T_C <  \tau_{D_r}])+\eta\,K. 
\end{align*}
So 
$$
	\frac{M}{K}\geq \frac{1-\eta-\zeta\,\P^{(x,s)}[ T_C< \tau_{D_r}]}{1-\P^{(x,s)}[ T_C< \tau_{D_r}]}>1.
$$
Then there is a $\beta>0$ such that $M>(1+2\beta)\,K$, and, as a result, there is $(x',s')\in D_{2r}$ such that $h(x',s')\geq (1+\beta )\,K$. 

Now suppose there is $(x_1,s_1)\in \widetilde{D}_1$ with $h(x_1,s_1)=K_1$. We can then find $r_1$ as above. And there is $(x_2,s_2)\in D_{2r_1}(x_1,s_1)$ with $h(x_2,s_2)=K_2\geq(1+\beta)\,K_1$. By induction we can create a sequence $\{(x_i,s_i)\}$ and corresponding $\{K_i\}$, $\{r_i\}$ so that $(x_{i+1},s_{i+1})\in D_{2r_i}(x_i,s_i)$ and $K_i\geq (1+\beta)^{i-1}K_1$. Note that 
$$
	\sum_i |x_{i+1}-x_{i}| \leq c_6 \,K_1^{-2/(2d+\alpha)} \quad \mbox{ and } \quad \sum_i |s_{i+1}-s_{i}| \leq c_7 \,K_1^{-\alpha/(2d+\alpha)}.
$$
So if $K_1 > (2c_6)^{d+\frac{\alpha}{2}}$ and $K_1 > (2c_7)^{\frac{2d}{\alpha}+1}$ then $(x_{i},s_{i})$'s are in $\widetilde{D}_2$, and $h(x_{i},s_{i})\geq (1+\beta)^{i-1}K_1$ which contradicts the fact that  $h$ is bounded. Hence $K_1\leq (2c_6)^{d+\frac{\alpha}{2}} \vee  (2c_7)^{\frac{2d}{\alpha}+1}=c$, and 
$$
	\sup_{\widetilde{D}_1}h\leq c.
$$
\end{proof}




\section{Regularity Results}\label{regularity}

One of the important application of the Harnack inequality is the regularity for the solution of elliptic PDEs. Bass and Levin [\ref{bass-levin}] develop some techniques to prove H\"older continuity and regularity results of some integral operators. They applied these techniques in the case of a jump Markov process whose kernel is comparable to that of a symmetric stable process. We will follow their method and show regularity results for our $\lambda$-resolvent.

We call the operator 
$$U_\lambda f(x,t)=\E^{(x,t)} \int_0^\infty e^{-\lambda s} f(X_s) ds=\int_0^\infty e^{-\lambda s} Q_sf(x,t)ds$$
as the $\lambda$-resolvent of $f$. Resolvent plays an important role in applications. We study some properties of this resolvent.

\begin{theorem}
If $f$ is a bounded function on $\Rd \times \Rp$ and it is harmonic on $D_4(x,t)\subset\Rd\times\Rp$, then $f$ is H\"older continuous in $D_1(x,t)$, that is, there exists $\gamma >0$ such that
\begin{align}\label{holder}
	\displaystyle |f(y,s)-f(y',s')| \leq c || f ||_\infty |(y,s)-(y',s')|^\gamma , \qquad (y,s),(y',s') \in D_1(x,t).
\end{align}
\end{theorem}
\begin{proof}
	Without loss of generality by taking a constant multiple, we may assume that $|| f ||_\infty=1$. First, we pick two points $(y,s),(y',s') \in D_1(x,t)$. Observing that $D_2(y,s) \subset D_4(x,t)$, $f$ is harmonic on $D_2(y,s)$ by the hypothesis. It is enough to consider the case when $(y',s') \in D_1(y,s)$. Because, otherwise, $ |(y,s)-(y',s')| \geq 1$ and hence
	$$|f(y,s)-f(y',s')| \leq 2 \leq 2  |(y,s)-(y',s')|^\gamma.$$
	By corollary \ref{harnack_corollary}, there is $c_0>0$ such that for any compact set $A \subset D_{r/3}(y,s)$ with $|A| \geq |D_{r/3}(y,s)|/3$ and $r\leq 1/3$ (so that $D_{2r}(y,s)\subset \Rd\times\Rp$)
	\begin{align}
		\P^{(y,s)}(T_A < \tau_{D_{r}(y,s)})\geq c_0.
	\end{align}
	We fix $c_0$ and choose $\beta \in (0,1)$  close enough to 1 so that $\beta^2 \geq (1-c_0/4) >0.$  Arguing as in lemma \ref{harnack_lemma_1}, if we take $H=\ind_{[\Rd\times\Rp - D_{r'}(y,s)]}$ such that $r' \geq 2r$ then for $u>0$
	\begin{align*}
		\P^{(y,s)}(X_{\tau_{D_{r}(y,s)} \wedge u} \not\in D_{r'}(y,s)) & = \E^{(y,s)}\left[ H(X_{\tau_{D_{r}(y,s)} \wedge u} ) \right] \\
			& \leq \E^{(y,s)}\left[ \sum_{v\leq \tau_{D_{r}(y,s)} \wedge u}  \ind_{\{ Y_v \not= Y_{v-}, Y_v \not\in \overrightarrow{D}_{r'}(y,s) \}} \right] .
	\end{align*}
	By Levy system formula and the fact that $|Y_v-h|$ is bounded below for $h\not\in \overrightarrow{D}_{r'}(y,s)$ and $Y_v \in \overrightarrow{D}_{r}(y,s)$, the last expression is bounded by
	\begin{align*}
		\E^{(y,s)} \int_0^{\tau_{D_{r}(y,s)} \wedge u} \int_{\Rd\times\Rp - D_{r'}(y,s)} \frac{1}{|h-Y_v|^{d+\alpha}} \, dh \, dv \leq c (r')^{-2} \E^{(y,s)}[\tau_{D_{r}(y,s)} \wedge u].
	\end{align*}
	Let $u\rightarrow \infty$. By lemma \ref{exp_leq_r2}, we obtain
	\begin{align}\label{ratio_of_boxes}
		\P^{(y,s)}(X_{\tau_{D_{r}(y,s)} \wedge u} \not\in D_{r'}(y,s)) & \leq c_1 \, \left(\frac{r}{r'} \right)^2.
	\end{align}
	We note that this inequality is valid even if $r'$ is large and $D_{r'}(y,s)\not\subset\Rd\times\Rp$, since the event considers only the jumps of $X_t$ from $D_r(y,s)$ to $\Rd\times\Rp-D_{r'}(y,s)$ and a jump may occur only in the horizontal direction. Now let $$\theta=\left(\frac{\beta}{2} \right)^{1/2}\wedge \left(\frac{\beta c_0}{8 c_1} \right)^{1/2} \wedge \left(\frac{1}{3} \right)$$
	and denote infimum and supremum over nested rectangular boxes by
	$$a_k=\inf_{D_{\theta^k}(y,s)} f \quad \mbox{and} \quad b_k=\sup_{D_{\theta^k}(y,s)} f .$$
	We will show that $b_k-a_k\leq \beta^k$ by induction. First of all, it is clear that the base step $b_0-a_0\leq \beta^0$ holds. Assume this inequality also holds for any $i\leq k$, that is  $b_i-a_i\leq \beta^i$. Now let $A'=\{f\leq (a_{k+1}+b_{k+1})/2\} \cap D_{\theta^{k+1}}(y,s)$. We may assume $|A'|\geq |D_{\theta^{k+1}}(y,s)|/2$. Otherwise we can work with $1-f$ instead of $f$. Take a compact set $A\subset A'$ so that $|A|\geq |D_{\theta^{k+1}}(y,s)|/3$.\\

	 Let  $\epsilon>0$ and pick $(z,u), (z',u') \in {D_{\theta^{k+1}}(y,s)}$ so that
	$$f(z,u) < a_{k+1}+\epsilon \quad \mbox{and} \quad f(z',u') > b_{k+1}-\epsilon .$$
	Since $f$ is harmonic in $D_{2}(y,s)$,
	\begin{align*}
		f(z,u)-f(z',u')&=\E^{(z,u)}\left[ f(X_{\tau_{D_{\theta^k}(y,s)}}) -f(z',u') \right] \\
				&=I_1+I_2+I_3.
	\end{align*} 
	where
	\begin{align*}
		I_1&=\E^{(z,u)}\left[ f(X_{\tau_{A}}) -f(z',u')  ;  \tau_A <  \tau_{D_{\theta^k}(y,s)} \right] \\
		I_2&= \E^{(z,u)}\left[ f(X_{ \tau_{D_{\theta^k}(y,s)}}) -f(z',u')  ;  \tau_A >  \tau_{D_{\theta^k}(y,s)} , X_{\tau_{D_{\theta^k}(y,s)}} \in D_{\theta^{k-1}}(y,s)\right] \\
		I_3&= \sum_{i=1}^\infty \E^{(z,u)}\left[ f(X_{ \tau_{D_{\theta^k}(y,s)}}) -f(z',u')  ;  \tau_A >  \tau_{D_{\theta^k}(y,s)} ,\right. \\
			&\qquad \qquad \qquad \qquad \qquad \qquad \qquad  \left.X_{\tau_{D_{\theta^k}(y,s)}} \in D_{\theta^{k-1-i}}(y,s)- D_{\theta^{k-i}}(y,s)\right]  .
	\end{align*}
	We will bound each term in terms of powers of $\beta$. First we note 
	\begin{align*}
		I_1&\leq \left(\frac{a_{k+1}+b_{k+1}}{2}-a_{k+1}\right) \P^{(z,u)}\left(\tau_A <  \tau_{D_{\theta^k}(y,s)} \right)  \\
			& \leq \left(\frac{b_{k-1}-a_{k-1}}{2} \right) \,   \P^{(z,u)}\left(\tau_A <  \tau_{D_{\theta^k}(y,s)} \right)  \\
			&\leq  \frac{\beta^{k-1}}{2}\, \P^{(z,u)}\left(\tau_A <  \tau_{D_{\theta^k}(y,s)} \right) 
	\end{align*}
	and
	\begin{align*}
		I_2& \leq \left({b_{k-1}-a_{k-1}} \right) \,   \P^{(z,u)}\left(\tau_A >  \tau_{D_{\theta^k}(y,s)} \right)  \\
			&\leq \beta^{k-1}\, \P^{(z,u)}\left(\tau_A >  \tau_{D_{\theta^k}(y,s)} \right) .
	\end{align*}
	Hence
	\begin{align}\label{I1_I2}
		I_1+I_2&\leq\beta^{k-1}\left[\frac{1}{2}\P^{(z,u)}\left(\tau_A <  \tau_{D_{\theta^k}(y,s)} \right) +1-\P^{(z,u)}\left(\tau_A <  \tau_{D_{\theta^k}(y,s)} \right) \right] \notag \\
			&\leq \beta^{k-1}(1-\frac{c_0}{2}).
	\end{align}
	For the third integral, we will use (\ref{ratio_of_boxes}), the fact that $\theta^2 \leq \beta/2$ so that $(1-(\theta^2/\gamma))^{-1}\leq 2$ and $\theta^2 \leq (\beta c_0)/(8c_1)$. Then
	\begin{align}\label{I3}
		I_3&\leq \sum_{i=1}^\infty  \left({b_{k-1-i}-a_{k-1-i}} \right)\,  \P^{(z,u)}\left( X_{\tau_{D_{\theta^k}(y,s)}} \in D_{\theta^{k-1-i}}(y,s)- D_{\theta^{k-i}}(y,s)\right)  \notag \\
		&\leq \sum_{i=1}^\infty  \beta^{k-1-i}\,  \P^{(z,u)}\left( X_{\tau_{D_{\theta^k}(y,s)}} \in D_{\theta^{k-1-i}}(y,s)- D_{\theta^{k-i}}(y,s)\right) \notag \\
		&\leq c_1 \sum_{i=1}^\infty  \beta^{k-1-i}\, \left( \frac{\theta^k}{\theta^{k-i}} \right)^2\notag \\
		&= c_1 \,  \beta^{k-1}\,\sum_{i=1}^\infty  \left( \frac{\theta^2}{\beta} \right)^i\\
		&\leq 2c_1\beta^{k-2} \theta^2\notag \\
		&\leq (c_0/4)\beta^{k-1}.\notag 
	\end{align}
	So by (\ref{I1_I2}) and (\ref{I3}),
	\begin{align*}
		f(z,u)-f(z',u')&\leq \beta^{k-1} (1-\frac{c_0}{4})\leq\beta^{k+1}.
	\end{align*} 
	This shows that $b_{k+1}-a_{k+1}\leq \beta^{k+1}$. \\
	Finally, fix $k$ so that $(y',s') \in D_{\theta^{k}}(y,s)- D_{\theta^{k+1}}(y,s) $. We observe that
	$$|(y,s)-(y',s')|\geq(\theta^{k+1}\wedge [\theta^{k+1}]^{2/\alpha})= [\theta^{k+1}]^{2/\alpha}$$
	since $(y',s') \not\in D_{\theta^{k+1}}(y,s) $ and $2/\alpha>1$. Hence $$\log{|(y,s)-(y',s')|} \geq (k+1)(2/ \alpha) \log(\theta).$$ Then
	\begin{align*}
		|f(y,s)-f(y',s')| & \leq e^{k \log(\beta)} \\ &\leq  e^{(\alpha/2)\log(\beta)\log{|(y,s)-(y',s')|}/ \log(\theta) } \\&=|(y,s)-(y',s')|^\gamma
	\end{align*}
	where $\gamma=(\alpha/2)\log(\beta)/ \log(\theta)$. This is the desired result.
\end{proof}

Having the H\"older continuity of harmonic function, we can discuss the continuity of the resolvent. As we defined before, $\lambda$-resolvent of $f$ is given by
$$U_\lambda f(x,t)=\E^{(x,t)} \int_0^\infty e^{-\lambda s} f(X_s) ds.$$
We will show that 
\begin{align}\label{U_ineq} |U_\lambda f(y,s)-U_\lambda f(y',s')|\leq c ||f||_\infty(|(y,s)-(y',s')| \wedge 1)^\gamma.\end{align}
First, we recall the resolvent equation
\begin{align}\label{res_eqn}
	(\beta-\lambda) U_\lambda U_\beta = U_\lambda - U_\beta.
\end{align}
It is easy to verify this equation due to the semi-group property of $Q_s$. This resolvent inequality reduces (\ref{U_ineq}) to the case $\lambda=0$.  So we will prove the following theorem first, and then we prove (\ref{U_ineq}) by using this theorem.
\begin{theorem}
	Suppose $f$ is a bounded function on $\Rd\times\Rp$ with compact support. Then there exists $\gamma \in (0,1)$ such that
	$$|U_0f(y,s)-U_0f(y',s')| \leq c \left( ||f||_\infty + ||U_0f||_\infty \right)(|(y,s)-(y',s')| \wedge 1)^\gamma$$
\end{theorem}
\begin{proof}
 First of all, without loss of generality we may assume $D_4(y,s)\subset\Rd\times\Rp$ by scaling property. We observe that if $|(y,s)-(y',s')| \geq 1$, the result follows immediately. So we assume $|(y,s)-(y',s')| < 1$, that is $(y',s')\in D_1(y,s)$. By definition
\begin{align*}
	U_0f(y,s)=\E^{(y,s)} \int_0^\infty  f(X_t) dt
\end{align*}
which can be written as 
\begin{align*}
	U_0f(y,s)=\E^{(y,s)} \int_0^{\tau_{D_r(y,s)}}  f(X_t) dt + \E^{(y,s)} \int_{\tau_{D_r(y,s)}}^\infty  f(X_t) dt.
\end{align*}
$r<1$ will be chosen later. (We only assume $r<1$.) By strong Markov property, the second term is equal to
\begin{align*}
	\E^{(y,s)} \int_{0}^\infty  f(X_{t+\tau_{D_r(y,s)}}) dt & =\E^{(y,s)} \E^{(y,s)}\left[ \left.\int_{0}^\infty   f(X_{t+\tau_{D_r(y,s)}}) dt \right| \tau_{D_r(y,s)} \right]\\
			&=\E^{(y,s)}\left[  U_0f( X_{\tau_{D_r(y,s)}})\right].
\end{align*}
Similarly,
\begin{align*}
	U_0f(y',s')=\E^{(y',s')} \int_0^{\tau_{D_r(y,s)}}  f(X_t) dt + \E^{(y',s')} \left[  U_0f( X_{\tau_{D_r(y,s)}})\right].
\end{align*}
Hence
\begin{align*}
	|U_0f(y,s)-U_0f(y',s')| &\leq ||f||_\infty \left[ \E^{(y,s)} \tau_{D_r(y,s)} +\E^{(y',s')} \tau_{D_r(y,s)}  \right] + \\
			& \left|\E^{(y,s)} \left[  U_0f( X_{\tau_{D_r(y,s)}})\right]-  \E^{(y',s')} \left[  U_0f( X_{\tau_{D_r(y,s)}})\right]\right|.
\end{align*}
Note that the first term is bounded by $c \, r^2\, ||f||_\infty  $. We also note that the function $(z,u)\rightarrow\E^{(z,u)} \left[  U_0f( X_{\tau_{D_r(y,s)}})\right]$ is harmonic in $D_r(y,s)$. So it is H\"older continuous by previous theorem. Then the second term on the right is bounded by
 \begin{align*}
	 c ||U_0f||_\infty\left[ \frac{|(y,s)-(y',s')|}{r\wedge r^{2/\alpha}} \right]^\gamma
\end{align*}
after scaling.
Now let $r=|(y,s)-(y',s')|^{\alpha/4}$. Since $|(y,s)-(y',s')|<1$, we have $|(y,s)-(y',s')|^{1/2}\leq |(y,s)-(y',s')|^{\alpha/4}$, and hence 
\begin{align*}
	|U_0f(y,s)-U_0f(y',s')| &\leq c\,  ||f||_\infty|(y,s)-(y',s')|^{\alpha/2} +c' \, ||U_0f||_\infty |(y,s)-(y',s')|^{\gamma/2} \\
			&\leq c''\, (  ||f||_\infty+||U_0f||_\infty) |(y,s)-(y',s')|^{(\alpha \wedge \gamma)/2} .
\end{align*}
\end{proof}

\begin{theorem}
	Suppose $f$ is bounded and $\lambda>0$. Then we have
	$$|U_\lambda f(y,s)-U_\lambda f(y',s')| \leq c ||f||_\infty (|(y,s)-(y',s')| \wedge 1)^\gamma$$
\end{theorem}

\begin{proof}
It is enough to show this for positive and compactly supported functions. Suppose $f>0$ and $f$ has compact support. Define the function $g$ by
\begin{align}
	g(y,s)=f(y,s)-\lambda U_\lambda f(y,s).
\end{align}
Note that
$$\lambda U_0U_\lambda f=U_0 f- U_\lambda f$$
by resolvent equation (\ref{res_eqn}) and hence 
$$U_0 g= U_0 f -\lambda U_0 U_\lambda f=U_\lambda f.$$
These equations can be verified by direct calculation. Using this equality, we obtain
$$||U_0 g||_\infty=||U_\lambda f ||_\infty\leq c \, ||f||_\infty$$
and 
$$||g||_\infty\leq   ||f||_\infty + \lambda ||U_\lambda f||_\infty\leq c \, ||f||_\infty.$$
Finally, by previous theorem 
\begin{align*}
	|U_\lambda f(y,s)-U_\lambda f(y',s')| & = |U_0 g(y,s)-U_0 g(y',s')| \\
			& \leq c \, \left( ||g||_\infty +  ||U_0 g||_\infty \right)(1\wedge |(y,s)-(y',s')|)^\gamma \\
			& \leq c \, ||f||_\infty (1\wedge |(y,s)-(y',s')|)^\gamma.
\end{align*}
\end{proof}





\section{Littlewood-Paley Functions}\label{LpTheory}
\label{sec:1}
This short section is about Littlewood-Paley functions obtained from the $\alpha$-harmonic extension of a function $f$ on $\Rd$. After defining G-functions in this context we state some earlier results by P.A. Meyer \cite{meyer_long,meyer_long_2,meyer_retour}  and N. Varopoulos \cite{varopoulos} and then we prove a result on a partial G-function which is close to the area functional used in the classical context. Let's denote the general Littlewood-Paley function by $G_f$ and define it by 
\begin{align*}
              \disp G_f(x) & = \disp \left[ \int_0^\infty t\, g(x,t) dt\right]^{1/2} \\ \\
                                    & =  \disp\left[ \int_0^\infty t\, \int \frac{[f_t(x+h)-f_t(x)]^2}{|h|^{d+\alpha}} dh \,dt + \int_0^\infty t \left[ \frac{\partial}{\partial t} f(x,t) \right]^2 dt \right]^{1/2}  
\end{align*}
where $f_t(x)=Q_tf(x)$ as defined in the section of preliminaries. As one can easily see, the integrand $g(x,t)$ is the square function which was defined before in (\ref{quad_var}). It has two components: one corresponding to the vertical process (Brownian motion) and the other one corresponding to the horizontal process (symmetric $\alpha$-stable process). Let's denote these components by $G_f^{\uparrow}$ and $\overrightarrow{G}_f$, respectively. Explicitly, these two functions are
$$
\disp \overrightarrow{G}_f(x) =\disp\left[ \int_0^\infty t\, \int \frac{[f_t(x+h)-f_t(x)]^2}{|h|^{d+\alpha}} dh \,dt  \right]^{1/2} \, ,
$$
and
\begin{align*}
\disp G^{\uparrow}_f(x) & =\disp\left[ \int_0^\infty t\, \left[ \frac{\partial}{\partial t} f(x,t) \right]^2 \,dt  \right]^{1/2} \,. 
\end{align*}

Some results on $L^p$-norms of these $G$-functions are known. P.A. Meyer worked in \cite{meyer_long_2} with symmetric Markov processes  and proved an $L^p$ inequality for the case $p\geq 2$. When applied to our setup, we can state it as $||G_f||_p \leq c ||f||_p$ for $p\geq 2$. On the other hand N. Varopoulos showed in his work \cite{varopoulos}  that this inequality can be extended to $p>1$ for the Brownian component, that is  $||G^\uparrow_f||_p \leq c ||f||_p$ for $p>1$. However the extension of this inequality for  the general $G$-function is not possible which is pointed by M. Silverstein \cite{meyer_correction}. Here we discuss the part of the horizontal component on a parabolic-like domain. Although the extension of this $L^p$ inequality is not true for $\overrightarrow{G}_f$ when $p\in (1,2)$, we can obtain a partial result considering the operator only inside a parabolic-like domain. In the classical context, the area functional is defined in a similar way on a cone with a vertex at a point $x\in \Rd$. Since we have different scaling factors on each component, we consider a modification of this domain and study the part of $G$-function on the set $\{(x+h,t)\in \Rd \times \Rp:|h|<t^{2/\alpha}\}$. For this purpose, define $\overrightarrow{G}_{f,\alpha}$ as

\begin{align*}
	\overrightarrow{G}_{f,\alpha}(x)  &  =\disp\left[ \int_0^\infty t\, \int_{\{|h|< t^{2/\alpha}\}} (f_t(x+h)-f_t(x))^2 \, \frac{dh}{|h|^{d+\alpha}} \,dt  \right]^{1/2} .
\end{align*}
First we note that the harmonic extension can be expressed as a convolution of the function with an approximate identity. To see this, we observe that transition densities of the symmetric stable process $Y_t$ satisfy the relation  $p(st^2,0,y)=t^{-2d/\alpha}p(s,0,yt^{-2/\alpha})$ by the scaling property. Hence we can write
$$
	f_t(x)=\int f(x-y) \,  \phi_{t}(y)\, dy = f * \phi_{t} (x)
$$
where 
$$
	\phi(x)=\int_0^\infty p(s,0,x) \mu_1(ds)
$$
and
$$
	\phi_t(x)=t^{-2d/\alpha}\phi(x/t^{2/\alpha}).
$$
So $f_t(x) \leq c\, \mathcal{M}(f)(x)$ \cite[section 2.1]{grafakos} where $\mathcal{M}(f)$ is the Hardy-Littlewood maximal function. To see this, it is enough to note that $\phi$ is radially decreasing and its $L^1$-norm equals one. Since the transition density $p(s,0,x)$ is obtained from the characteristic function $e^{-s|x|^\alpha}$ by the inverse Fourier transform, it follows trivially that $p(s,0,x)$ is a radial symmetric function. Moreover, by \cite[P.261]{sato} $p(s,0,x)$ can also be expressed as $\int_0^\infty (4\pi u)^{-d/2} e^{-|x|^2/(4u)} g_{\alpha/2}(s,u)du$ where  $ g_{\alpha/2}(s,u)$ is the density of an $\alpha/2$ stable subordinator whose Laplace transform is given by $\int_0^\infty e^{-\lambda v}g_{\alpha/2}(s,v)dv=e^{-s\lambda^{\alpha/2}}$. Hence $p(s,0,x)$ is radially decreasing in the variable $x$. The fact that $\phi$ is radially decreasing follows from the previous line.

 It is also known that 
$$
	\|\mathcal{M}(f)\|_p \leq c \, \|f\|_p, \quad \mbox{when} \quad f\in L^p  \mbox{ and } p\in (1,\infty).
$$

\begin{lemma}\label{f_norm} If $f\in L^p$, then for some $\xi=\xi(x,h,t)$ between $f_t(x)$ and $f_t(x+h)$,
	$$||f||_p^p\geq c \, \int \int_0^\infty t\, \int \xi^{p-2} (f(x+h)-f_t(x))^2 \, \frac{dh}{|h|^{d+\alpha}} \,dt \, dx.$$
\end{lemma}
 \begin{proof}
     Suppose $f$ is positive and consider the map $F(x)=(x+\epsilon)^p$ for some $\epsilon >0$. Clearly, $F\in \mathcal{C}^2(\mathbb{R}^+)$. So we can apply the It\^o formula using the martingale $M^f_t=f(X_{t\wedge T_0})$, and we obtain
     \begin{align*}
     	(M^f_t+\epsilon)^p & = (M^f_0+\epsilon)^p + p \int_0^t (M^f_{s-}+\epsilon)^{p-1} dM^f_s\\
					&\qquad+\frac{p(p-1)}{2}\int_0^t (M^f_{s-}+\epsilon)^{p-2} d\langle (M^f)^c \rangle_s \\
					& \qquad +\sum_{s\leq t} \left[ (M^f_s+\epsilon)^p- (M^f_{s-}+\epsilon)^p -p (M^f_{s-}+\epsilon)^{p-1} (M^f_s-M^f_{s-})\right]. 
     \end{align*}
     By \cite[P.168]{meyer_long}, there is a positive  function $j(s)$ such that $d\langle (M^f)^c \rangle_s\geq j(s) ds$, and hence the third term on the right hand side is positive. Moreover, the convexity of the function $F$ implies that the jump terms are also positive. So taking expectations first, and letting $t\rightarrow \infty$ and $\epsilon \rightarrow 0$, we have
     \begin{align*}
     	\Ema((M^f_{T_0})^p )& \geq \Ema\left( \sum_{s\leq T_0} \left[ (M^f_s)^p- (M^f_{s-})^p -p (M^f_{s-})^{p-1} (M^f_s-M^f_{s-})\right] \right). 
     \end{align*}
    We recall that $||f||_p^p=\Ema((M^f_{T_0})^p )$ [Remarks (\ref{remark:4}) and (\ref{remark:5})]. If we denote by $\Lambda(f,x,y,z)$ the expression
\begin{align*}
	\Lambda(f,x,y,z)=f^p(x,z)- f^p(y,z) -p f^{p-1}(y,z) (f(x,z)-f(y,z))
\end{align*}
then we can write
    \begin{align*}
     	||f||_p^p& \geq \Ema\left( \sum_{s\leq T_0} \left[ \Lambda(f,Y_s,Y_{s-},Z_s)\right] \right). 
     \end{align*}
   Next, we will use the L\'evy system formula, invariance of $P_t$ under the Lebesgue measure $m$, and Green's function, respectively, and the right hand side of the last inequality becomes
   \begin{align*}
	              & \Ema\left( \int_0^{T_0} \int \Lambda(f,Y_s+h,Y_{s},Z_s) \frac{dh}{|h|^{d+\alpha}} \, ds \right) \\ \\
	              &= \int \E^a\left( \int_0^{T_0} \int\Lambda(f,y+h,y,Z_s)\frac{dh}{|h|^{d+\alpha}} \, ds \right) dy \\ \\
	              &= \int  \int_0^{\infty} (a\wedge t) \int\Lambda(f,y+h,y,t)\frac{dh}{|h|^{d+\alpha}} \, dt \, dy.  
     \end{align*}
     If we let $a\rightarrow \infty$, then
      \begin{align*}
	       ||f||_p^p  &\geq \int  \int_0^{\infty}  t \int \Lambda(f,y+h,y,t) \frac{dh}{|h|^{d+\alpha}} \, dt \, dy.  
     \end{align*}
     If we use the Taylor expansion of $x\rightarrow x^p$, then for some $\xi$ between $f_t(x)$ and $f_t(x+h)$ we have  
     $$\Lambda(f,y+h,y,t)= \frac{p(p-1)}{2}\xi^{p-2}(f_t(y+h)-f_t(y))^2, $$
     and using this equality we obtain the result
     $$||f||_p^p\geq c \, \int \int_0^\infty t\, \int \xi^{p-2} (f(x+h)-f_t(x))^2 \, \frac{dh}{|h|^{d+\alpha}} \,dt \, dx.$$
 \end{proof}
 \begin{theorem}\label{area_function}
 If $p\in(1,2)$ and $f\in L^p(\Rd)$ then 
 $$
 	\|\overrightarrow{G}_{f,\alpha}\|_p \leq c \, \|f\|_p.
 $$
 \end{theorem}
 \begin{proof}
 Without loss of generality we may assume that $f \geq \epsilon>0$. Then it can be generalized to $f\in L^p$. By definition of the harmonic extension, we also have $f_t\geq \epsilon$. By Lemma (\ref{f_norm}), there is $\xi=\xi(t,h,x)$ between $f_t(x)$ and $f_t(x+h)$ such that 
 $$||f||_p^p\geq c \, \int \int_0^\infty t\, \int \xi^{p-2} (f(x+h)-f_t(x))^2 \, \frac{dh}{|h|^{d+\alpha}} \,dt \, dx.$$
 Using this $\xi$, we can write
 \begin{align*}
 	\overrightarrow{G}_{f,\alpha}(x)  &  =\disp\left[ \int_0^\infty t\, \int_{\{|h|<t^{2/\alpha}\}} (f_t(x+h)-f_t(x))^2 \, \frac{dh}{|h|^{d+\alpha}} \,dt  \right]^{1/2} \\ \\
		    & = \disp\left[ \int_0^\infty t\, \int_{\{|h|<t^{2/\alpha}\}} \, \xi^{2-p}\, \xi^{p-2}\, (f_t(x+h)-f_t(x))^2 \, \frac{dh}{|h|^{d+\alpha}} \,dt  \right]^{1/2}.
 \end{align*}
 Temporarily, we fix $t>0$. Let $R_t$ be the rectangular box centered at $(x,t)$ with side-length $(\frac{t}{32})^{2/\alpha}, (\frac{t}{32})^{2/\alpha}, ... , (\frac{t}{32})^{2/\alpha}, \frac{t}{32}$, that is $R_t=D_{t/32}(x,t)$. Then $R_{32t} \subset \Rd\times\Rp$. Let $\beta(x,h,t)$ denotes the linear path from $(x,t)$ to $(x+h,t)$ where $|h|<t^{2/\alpha}$. This path can be covered by $n$-many horizontal translations of the box $R_t$, say $R^1_t, R^2_t, ... , R^n_t$, such that $R^j_t \cap R^{j+1}_t \not=\emptyset$ for $j\in\{1,2,…,n-1\}$ and $R_t \cap R^{1}_t \not=\emptyset$. Note that $n$ can be chosen so that it does not depend on $t$, $x$, or $h$. Choose points from each pairwise intersection, say $(x_1,t)\in R_t \cap R^1_t$, $(x_2,t)\in R^1_t \cap R^2_t$, ... , $(x_n,t)\in R^{n-1}_t \cap R^n_t$. By using the Harnack inequality, we obtain
	\begin{align*}
		f_t(x+h)\leq c \, f_t(x_n) \leq c^2 \, f_t(x_{n-1}) \leq ... \leq c^n f_t(x_1) \leq c^{n+1} f_t(x).
	\end{align*}
	Hence, if $|h|<t^{2/\alpha}$, then $f_t(x+h) \leq c^{n+1} f_t(x)$, and $\xi \leq [f_t(x) \vee f_t(x+h)]\leq c^{n+1} f_t(x)$. This implies that  $\xi^{2-p} \leq  \, c\,f_t^{2-p}(x)  \leq c' \, [\mathcal{M}(f)(x)]^{2-p} $. So $\overrightarrow{G}_{f,\alpha}(x)$ is bounded above by 
	 \begin{align*}
 	 c'\, \disp [\mathcal{M}(f)(x)]^{(2-p)/2} \left[ \int_0^\infty t\, \int_{\{|h|<t^{2/\alpha}\}} \,  \xi^{p-2}\, (f_t(x+h)-f_t(x))^2 \, \frac{dh}{|h|^{d+\alpha}} \,dt  \right]^{1/2}.
 \end{align*}	
 By using the H\"older inequality with 2/(2-p) and 2/p,
  \begin{align*}
 	 \|\overrightarrow{G}_{f,\alpha}(x)\|_p^p   & \leq c'  \int \disp [\mathcal{M}(f)(x)]^{(2-p)p/2} \cdot \\ & \qquad \left[ \int_0^\infty t\, \int_{\{|h|<t^{2/\alpha}\}} \,  \xi^{p-2}\, (f_t(x+h)-f_t(x))^2 \, \frac{dh}{|h|^{d+\alpha}} \,dt  \right]^{p/2} m(dx) \\
		& \leq c'\,  \left[ \int \disp [\mathcal{M}(f)(x)]^p m(dx) \right]^{(2-p)/2} \, \cdot \\ & \qquad  \left[ \int  \int_0^\infty t\, \int_{\{|h|<t^{2/\alpha}\}} \,  \xi^{p-2}\, (f_t(x+h)-f_t(x))^2 \, \frac{dh}{|h|^{d+\alpha}} \,dt  m(dx)\right]^{p/2}\\
		& \leq c'\, \|f\|_p^{p(2-p)/2} \|f\|_p^{p^2/2}.
 \end{align*}
Then the desired result follows.
 \end{proof}


\begin{acknowledgements}
The results of this paper are
part of my Ph.D.
dissertation \cite{karli} written under the supervision of Prof. Richard Bass at the Department of Mathematics of the University of Connecticut. I would like to thank him for his valuable comments and his support.
\end{acknowledgements}



\end{document}